\documentclass{amsart}
\usepackage{amsfonts}

\newtheorem{theorem}{Theorem}[section]
\newtheorem*{theorem A}{Theorem A}
\newtheorem*{theorem B}{N\"olker's Theorem}
\newtheorem{lemma}{Lemma}[section]

\theoremstyle{remark}
\newtheorem{remark}{Remark}[section]
\theoremstyle{remark}

\theoremstyle{definition}
\newtheorem{definition}{Definition}[section]

\numberwithin{equation}{section}
\def\({\left ( }
\def\){\right )}
\def\<{\left < }
\def\>{\right >}

\begin{document}

\noindent {\sc {Mathematical Sciences And Applications E-Notes}}

\noindent {\sc \small Volume 1  No. 1 pp. 000--000 (2013) \copyright
MSAEN}

\vspace{2cm}

\title{On Hadamard Type Integral Inequalities for nonconvex Functions}
\author{Mehmet Zeki Sarikaya, Hakan Bozkurt and Necmettin Alp}
\address{Department of Mathematics, \ Faculty of Science and Arts, D\"{u}zce
University, D\"{u}zce-TURKEY}
\email{sarikayamz@gmail.com, insedi@yahoo.com and placenn@gmail.com}

\date{January 1, 2013 and, in revised form, June 22, 2013.}
\dedicatory{\textrm{(Communicated by Murat TOSUN)}}

\begin{abstract}
In this paper, we extend some estimates of the right and left hand side of a
Hermite- Hadamard type inequality for nonconvex functions whose derivatives
absolute values are $\varphi $-convex and quasi-$\varphi $-convex was
introduced by Noor in \cite{Noor1}.
\end{abstract}
\maketitle

\section{Introduction}

It is well known that if $f$ is a convex function on the interval $I=\left[
a,b\right] $ and $a,b\in I$ with $a<b$, then%
\begin{equation}
f\left( \frac{a+b}{2}\right) \leq \frac{1}{b-a}\int\limits_{a}^{b}f\left(
x\right) dx\leq \frac{f\left( a\right) +f\left( b\right) }{2}.  \label{H}
\end{equation}%
which is known as the Hermite-Hadamard inequality for the convex functions.

Both inequalities hold in the reversed direction if $f$ is concave. We note
that Hadamard's inequality may be regarded as a refinement of the concept of
convexity and it follows easily from Jensen's inequality. Hadamard's
inequality for convex functions has received renewed attention in recent
years and a remarkable variety of refinements and generalizations have been
found (see, for example, \cite{SSDRPA}, \cite{Dragomir2}, \cite{USK}-\cite{K}%
, \cite{CEMPJP}, \cite{sarikaya}-\cite{yildiz}).

In \cite{USK} some inequalities of Hermite-Hadamard type for differentiable
convex mappings connected with the left part of (\ref{H}) were proved by
using the following lemma:

\begin{lemma}
\label{l1} Let $f:I^{\circ }\subset \mathbb{R}\rightarrow \mathbb{R}$, be a
differentiable mapping on $I^{\circ }$, $a,b\in I^{\circ }$ ($I^{\circ }$ is
the interior of $I$) with $a<b$. If \ $f^{\prime }\in L\left( \left[ a,b%
\right] \right) $, then we have%
\begin{equation}
\begin{array}{l}
\dfrac{1}{b-a}\int_{a}^{b}f(x)dx-f\left( \dfrac{a+b}{2}\right) \\
\\
\ \ \ \ \ =\left( b-a\right) \left[ \int_{0}^{\frac{1}{2}}tf^{\prime
}(ta+(1-t)b)dt+\int_{\frac{1}{2}}^{1}\left( t-1\right) f^{\prime
}(ta+(1-t)b)dt\right] .%
\end{array}
\label{HH}
\end{equation}
\end{lemma}

One more general result related to (\ref{HH}) was established in \cite%
{USKMEO}. The main result in \cite{USK} is as follows:

\begin{theorem}
\label{t1} Let $f:I\subset \mathbb{R}\rightarrow \mathbb{R}$, be a
differentiable mapping on $I^{\circ }$, $a,b\in I$ with $a<b$. If the
mapping $\left\vert f^{\prime }\right\vert $ is convex on $\left[ a,b\right]
$, then%
\begin{equation}
\left\vert \frac{1}{b-a}\int_{a}^{b}f(x)dx-f\left( \frac{a+b}{2}\right)
\right\vert \leq \frac{b-a}{4}\left( \frac{\left\vert f^{\prime
}(a)\right\vert +\left\vert f^{\prime }(b)\right\vert }{2}\right) .
\label{H1}
\end{equation}
\end{theorem}

In \cite{SSDRPA}, Dragomir and Agarwal established the following results
connected with the right part of (\ref{H}) as well as to apply them for some
elementary inequalities for real numbers and numerical integration:

\begin{theorem}
Let $f:I^{\circ }\subset \mathbb{R}\rightarrow \mathbb{R}$ be a
differentiable mapping on $I^{\circ }$, $a,b\in I^{\circ }$ with $a<b,$ and $%
f^{\prime }\in L(a,b).$ If the mapping $\left\vert f^{\prime }\right\vert $
is convex on $\left[ a,b\right] $, then the following inequality holds:%
\begin{equation}
\left\vert \dfrac{f(a)+f(b)}{2}-\dfrac{1}{b-a}\int_{a}^{b}f(x)dx\right\vert
\leq \left( b-a\right) \left( \frac{\left\vert f^{\prime }(a)\right\vert
+\left\vert f^{\prime }(b)\right\vert }{8}\right) .  \label{1H}
\end{equation}
\end{theorem}

In \cite{CEMPJP}, Pearce and Pe\v{c}ari\'{c} proved the following theorem:

\begin{theorem}
Let $f:I\subset \mathbb{R}\rightarrow \mathbb{R}$, be a differentiable
mapping on $I^{\circ }$, $a,b\in I^{\circ }$ with $a<b$. If the mapping $%
\left\vert f^{\prime }\right\vert ^{q}$ is convex on $\left[ a,b\right] $
for some $q\geq 1$, then%
\begin{equation}
\left\vert \dfrac{f(a)+f(b)}{2}-\dfrac{1}{b-a}\int_{a}^{b}f(x)dx\right\vert
\leq \frac{b-a}{4}\left( \frac{\left\vert f^{\prime }(a)\right\vert
^{q}+\left\vert f^{\prime }(b)\right\vert ^{q}}{2}\right) ^{\frac{1}{q}}
\label{2H}
\end{equation}%
and%
\begin{equation}
\left\vert \frac{1}{b-a}\int_{a}^{b}f(x)dx-f\left( \frac{a+b}{2}\right)
\right\vert \leq \frac{b-a}{4}\left( \frac{\left\vert f^{\prime
}(a)\right\vert ^{q}+\left\vert f^{\prime }(b)\right\vert ^{q}}{2}\right) ^{%
\frac{1}{q}}.  \label{3H}
\end{equation}
\end{theorem}

We recall that the notion of quasi-convex functions generalizes the notion
of convex functions. More precisely, a function $f:[a,b]\subset \mathbb{R}%
\rightarrow \mathbb{R}$ is said quasi-convex on $[a,b]$ if
\begin{equation*}
f(tx+(1-t)y)\leq \sup \left\{ f(x),f(y)\right\}
\end{equation*}%
for all $x,y\in \lbrack a,b]$ and $t\in \left[ 0,1\right] .$ Clearly, any
convex function is a quasi-convex function. Furthermore, there exist
quasi-convex functions which are not convex (see \cite{ion}).

The classical Hermite-Hadamard inequality provides estimates of the mean
value of a continuous convex function $f:[a,b]\rightarrow \mathbb{R}$. Ion
in \cite{ion} presented some estimates of the right hand side of a Hermite-
Hadamard type inequality in which some quasi-convex functions are involved.
The main results of \cite{ion} are given by the following theorems.

\begin{theorem}
Assume $a,b\in \mathbb{R}$ with $a<b$ and $f:[a,b]\rightarrow \mathbb{R}$ is
a differentiable function on $(a,b)$. If $\left\vert f^{\prime }\right\vert $
is quasi-convex on $[a,b],$ then the following inequality holds%
\begin{equation*}
\left\vert \dfrac{f(a)+f(b)}{2}-\dfrac{1}{b-a}\int_{a}^{b}f(x)dx\right\vert
\leq \frac{\left( b-a\right) }{4}\sup \left\{ \left\vert f^{\prime
}(a)\right\vert ,\left\vert f^{\prime }(b)\right\vert \right\} .
\end{equation*}
\end{theorem}

\begin{theorem}
Assume $a,b\in \mathbb{R}$ with $a<b$ and $f:[a,b]\rightarrow \mathbb{R}$ is
a differentiable function on $(a,b)$. Assume $p\in
\mathbb{R}
$ with $p>1$. If $\left\vert f^{\prime }\right\vert ^{p/p-1}$ is
quasi-convex on $[a,b],$ then the following inequality holds%
\begin{equation*}
\left\vert \dfrac{f(a)+f(b)}{2}-\dfrac{1}{b-a}\int_{a}^{b}f(x)dx\right\vert
\leq \frac{\left( b-a\right) }{2(p+1)^{\frac{1}{p}}}\left[ \sup \left\{
\left\vert f^{\prime }(a)\right\vert ^{p/p-1},\left\vert f^{\prime
}(b)\right\vert ^{p/p-1}\right\} \right] ^{\frac{p-1}{p}}.
\end{equation*}
\end{theorem}

Convexity plays a central and fundamental role in mathematical finance,
economics, engineering, management sciences and optimizastion theory. In
recent years, several extensions and generalizations have been considered
for classical convexity. A significant generalization of convex functions is
that of $\varphi $-convex functions introduced by Noor in \cite{Noor1}. In
\cite{Noor1} and \cite{Noor5}, the authors have studied the basic properties
of the $\varphi $-convex functions. It is well-know that the $\varphi $%
-convex functions and $\varphi $-sets may not be convex functions and convex
sets. This class of nonconvex functions include the classical convex
functions and its various classes as special cases. For some recent results
related to this nonconvex functions, see the papers \cite{Noor1}-\cite{Noor5}

\section{Preliminaries}

Let $f,\varphi :K\rightarrow
\mathbb{R}
^{n}$, where $K$ is a nonempty closed set in $%
\mathbb{R}
^{n}$, be continuous functions. First of all, we recall the following well
know results and concepts, which are mainly due to Noor and Noor \cite{Noor1}%
\ and Noor \cite{Noor5} as follows:

\begin{definition}
\label{d1} Let $u,v\in K$. Then the set $K$ is said to be $\varphi -convex$
at $u$ with respect to $\varphi $, if%
\begin{equation*}
u+te^{i\varphi }\left( v-u\right) \in K,\text{ }\forall u,v\in K,\text{ }%
t\in \left[ 0,1\right] .
\end{equation*}

\begin{remark}
\label{r1} We would like to mention that the Definition \ref{d1} of a $%
\varphi -convex$ set has a clear geometric interpretation. This definition
essentially says that there is a path starting from a point $u$ which is
contained in $K$. We do not require that the point $v$ should be one of the
end points of the path. This observation plays an important role in our
analysis. Note that, if we demand that $v$ should be an end point of the
path for every pair of points, $u,v\in K$, then $e^{i\varphi }\left(
v-u\right) =v-u$ if and only if, $\varphi =0$, and consequently $\varphi
-convexity$ reduces to convexity. Thus, it is true that every convex set is
also an $\varphi -convex$ set, but the converse is not necessarily true, see
\cite{Noor1},\cite{Noor5} and the references therein.
\end{remark}
\end{definition}

\begin{definition}
\label{d2} The function $f$ on the $\varphi -convex$ set $K$ is said to be $%
\varphi -convex$ with respect to $\varphi $, if%
\begin{equation*}
f\left( u+te^{i\varphi }\left( v-u\right) \right) \leq \left( 1-t\right)
f\left( u\right) +tf\left( v\right) ,\text{ }\forall u,v\in K,\text{ }t\in %
\left[ 0,1\right] .
\end{equation*}%
The function $f$ is said to be $\varphi -concave$ if and only if $-f$ is $%
\varphi -convex$. Note that every convex function is a $\varphi -convex$
function, but the converse is not true. For eample, the function $f:\mathbb{%
R\rightarrow }\mathbb{R}$, $f(x)=-\left\vert x\right\vert $ is not a convex
function$,$ but $f(x)=-\left\vert x\right\vert $ is a $\varphi -convex$ with
respect to $\varphi $ where%
\begin{equation*}
\varphi (v,u)=\left\{
\begin{array}{c}
2k\pi ,\text{ \ \ \ \ \ }u.v\geq 0,\ k\in \mathbb{Z} \\
k\pi ,\text{ \ \ \ \ \ \ }u.v<0,\ k\in \mathbb{Z}\text{.}%
\end{array}%
\right.
\end{equation*}
\end{definition}

\begin{definition}
\label{d3} The function $f$ on the $\varphi -convex$ set $K$ is said to be
logarithmic $\varphi -convex$ with respect to $\varphi $, such that%
\begin{equation*}
f\left( u+te^{i\varphi }\left( v-u\right) \right) \leq \left( f\left(
u\right) \right) ^{1-t}\left( f\left( v\right) \right) ^{t},\text{ }u,v\in K,%
\text{ }t\in \left[ 0,1\right]
\end{equation*}%
where $f\left( .\right) >0$.
\end{definition}

Now we define a new definition for quasi-$\varphi -convex$ functions as
follows:

\begin{definition}
\label{d4} The function $f$ on the quasi $\varphi -convex$ set $K$ is said
to be quasi $\varphi -convex$ with respect to $\varphi $, if%
\begin{equation*}
f\left( u+te^{i\varphi }\left( v-u\right) \right) \leq \max \left\{ f\left(
u\right) ,f\left( v\right) \right\} .
\end{equation*}
\end{definition}

From the above definitions, we have%
\begin{eqnarray*}
f\left( u+te^{i\varphi }\left( v-u\right) \right) &\leq &\left( f\left(
u\right) \right) ^{1-t}\left( f\left( v\right) \right) ^{t} \\
&\leq &\left( 1-t\right) f\left( u\right) +tf\left( v\right) \\
&\leq &\max \left\{ f\left( u\right) ,f\left( v\right) \right\} .
\end{eqnarray*}%
In \cite{Noor3},\ Noor proved the Hermite-Hadamard inequality for the $%
\varphi -$convex functions as follows:

\begin{theorem}
\label{tt1} Let $f:K=\left[ a,a+e^{i\varphi }\left( b-a\right) \right]
\rightarrow \left( 0,\infty \right) $ be a $\varphi -convex$\ function on
the interval of real numbers $K^{0}$ (the interior of $K$) and $a,b\in K^{0}$
with $a<a+e^{i\varphi }\left( b-a\right) $ and $0\leq \varphi \leq \frac{\pi
}{2}$. Then the following inequality holds:%
\begin{eqnarray}
f\left( \frac{2a+e^{i\varphi }\left( b-a\right) }{2}\right) &\leq &\frac{1}{%
e^{i\varphi }\left( b-a\right) }\int\limits_{a}^{a+e^{i\varphi }\left(
b-a\right) }f\left( x\right) dx  \label{2} \\
&\leq &\frac{f\left( a\right) +f\left( a+e^{i\varphi }\left( b-a\right)
\right) }{2}\leq \frac{f\left( a\right) +f\left( b\right) }{2}.  \notag
\end{eqnarray}
\end{theorem}

This inequality can easily show that using the $\varphi -convex$\ function's
definition and $f\left( a+e^{i\varphi }\left( b-a\right) \right) <f\left(
b\right) .$

In this article, using functions whose derivatives absolute values are $%
\varphi $-convex and quasi-$\varphi $-convex, we obtained new inequalities
releted to the right and left side of Hermite-Hadamard inequality. In
particular if $\varphi =0$ is taken as, our results obtained reduce to the
Hermite-Hadamard type inequality for classical convex functions.

Throughout this study, we always assume that $K=\left[ a,a+e^{i\varphi }(b-a)%
\right] $ and $0\leq \varphi \leq \frac{\pi }{2}$ the interval, unless
otherwise specified.

We shall start with the following refinements of the Hermite-Hadamard
inequality for $\varphi -$convex functions. Firstly, we give the following
results connected with the right part of (\ref{2}):

\begin{theorem}
\label{tt2} Let $f:K\rightarrow (0,\infty )$ be a differentiable mapping on $%
K^{0}$. If $\left\vert f^{\prime }\right\vert $ is $\varphi -$convex
function on the interval of real numbers $K^{0}$ (the interior of K) and $%
a,b\in K$ with $a<a+e^{i\varphi }(b-a)$. Then, the following inequality
holds:%
\begin{eqnarray}
&&\left\vert \frac{1}{e^{i\varphi }(b-a)}\int_{a}^{a+e^{i\varphi
}(b-a)}f(x)dx-\frac{f(a)+f(a+e^{i\varphi }(b-a))}{2}\right\vert  \notag \\
&&  \label{3} \\
&\leq &\frac{e^{i\varphi }(b-a)}{8}\left[ \left\vert f^{\prime
}(a)\right\vert +\left\vert f^{\prime }(b)\right\vert \right] .  \notag
\end{eqnarray}

\begin{proof}
Since $K$ is $\varphi -$convex with respect to $\varphi $, for every $t\in %
\left[ 0,1\right] $, we have $a+te^{i\theta }(b-a)\in K$. Integrating by
parts implies that%
\begin{eqnarray}
&&\int_{0}^{1}(1-2t)f^{\prime }(a+te^{i\varphi }(b-a))dt  \notag \\
&&  \label{4} \\
&=&\left[ \frac{(1-2t)f(a+te^{i\varphi }(b-a))}{e^{i\varphi }(b-a)}\right]
_{0}^{1}+\frac{2}{e^{i\varphi }(b-a)}\int_{0}^{1}f(a+te^{i\varphi }(b-a))dt
\notag \\
&&  \notag \\
&=&-\frac{f(a)+f(a+e^{i\varphi }(b-a))}{e^{i\varphi }(b-a)}+\frac{2}{%
e^{2i\varphi }(b-a)^{2}}\int_{a}^{a+e^{i\varphi }(b-a)}f(x)dx  \notag
\end{eqnarray}%
By $\varphi -$convexity\ of\ $\left\vert f^{\prime }\right\vert $ and (\ref%
{4}), we have%
\begin{eqnarray*}
&&\left\vert \frac{1}{e^{i\varphi }(b-a)}\int_{a}^{a+e^{i\varphi
}(b-a)}f(x))dx-\frac{f(a)+f(a+e^{i\varphi }(b-a))}{2}\right\vert  \\
&=&\frac{e^{i\varphi }(b-a)}{2}\left\vert \int_{0}^{1}(1-2t)f^{\prime
}(a+te^{i\varphi }(b-a))dt\right\vert  \\
&\leq &\frac{e^{i\varphi }(b-a)}{2}\int_{0}^{1}\left\vert 1-2t\right\vert %
\left[ (1-t)\left\vert f^{\prime }(a)\right\vert +t\left\vert f^{\prime
}(b)\right\vert \right] dt \\
&=&\frac{e^{i\varphi }(b-a)}{8}\left[ \left\vert f^{\prime }(a)\right\vert
+\left\vert f^{\prime }(b)\right\vert \right] .
\end{eqnarray*}%
which completes the proof.
\end{proof}
\end{theorem}

\begin{theorem}
\label{tt3}Let $f:K\rightarrow (0,\infty )$ be a differentiable mapping on $%
K^{0}$. Assume $p\in
\mathbb{R}
$ with $p>1$. If $\left\vert f^{\prime }\right\vert ^{p/p-1}$ is $\varphi -$%
convex function on the interval of real numbers $K^{0}$ (the interior of $K$%
) and $a,b\in K$ with $a<a+e^{i\varphi }(b-a)$. Then, the following
inequality holds:%
\begin{eqnarray}
&&\left\vert \frac{1}{e^{i\varphi }(b-a)}\int_{a}^{a+e^{i\varphi
}(b-a)}f(x)dx-\frac{f(a)+f(a+e^{i\varphi }(b-a))}{2}\right\vert  \notag \\
&&  \label{5} \\
&\leq &\frac{e^{i\varphi }(b-a)}{2(p+1)^{\frac{1}{p}}}\left( \frac{%
\left\vert f^{\prime }(a)\right\vert ^{\frac{p}{p-1}}+\left\vert f^{\prime
}(b)\right\vert ^{\frac{p}{p-1}}}{2}\right) ^{\frac{p-1}{p}}  \notag
\end{eqnarray}

\begin{proof}
From H\"{o}lder's inequality and by using (\ref{4}) in the proof of Theorem %
\ref{tt2}, we have%
\begin{eqnarray*}
&&\left\vert \frac{1}{e^{i\varphi }(b-a)}\int_{a}^{a+e^{i\varphi
}(b-a)}f(x))dx-\frac{f(a)+f(a+e^{i\varphi }(b-a))}{2}\right\vert  \\
&\leq &\frac{e^{i\varphi }(b-a)}{2}\int_{0}^{1}\left\vert 1-2t\right\vert
\left\vert f^{\prime }(a+te^{i\varphi }(b-a))\right\vert dt \\
&\leq &\frac{e^{i\varphi }(b-a)}{2}\left( \int_{0}^{1}\left\vert
1-2t\right\vert ^{p}dt\right) ^{\frac{1}{p}}\left( \int_{0}^{1}\left\vert
f^{\prime }(a+te^{i\varphi }(b-a))\right\vert ^{\frac{p}{p-1}}dt\right) ^{%
\frac{p-1}{p}} \\
&\leq &\frac{e^{i\varphi }(b-a)}{2}\left( \int_{0}^{1}\left\vert
1-2t\right\vert ^{p}dt\right) ^{\frac{1}{p}}\left( \int_{0}^{1}\left[
(1-t)\left\vert f^{\prime }(a)\right\vert ^{\frac{p}{p-1}}+t\left\vert
f^{\prime }(b)\right\vert ^{\frac{p}{p-1}}\right] dt\right) ^{\frac{p-1}{p}}
\\
&=&\frac{e^{i\varphi }(b-a)}{2(p+1)^{\frac{1}{p}}}\left( \frac{\left\vert
f^{\prime }(a)\right\vert ^{\frac{p}{p-1}}+\left\vert f^{\prime
}(b)\right\vert ^{\frac{p}{p-1}}}{2}\right) ^{\frac{p-1}{p}}.
\end{eqnarray*}%
This implies ineqyality (\ref{5}).
\end{proof}
\end{theorem}

Now, we give the following results connected with the left part of (\ref{2}):

\begin{theorem}
\label{tt4} Under the assumptions of Theorem \ref{tt2}. Then the following
inequality holds:%
\begin{eqnarray}
&&\left\vert \frac{1}{e^{i\varphi }(b-a)}\int_{a}^{a+e^{i\varphi
}(b-a)}f(x)dx-f\left( \frac{2a+e^{i\varphi }(b-a)}{2}\right) \right\vert
\notag \\
&&  \label{6} \\
&\leq &\frac{e^{i\varphi }(b-a)}{8}\left[ \left\vert f^{\prime
}(a)\right\vert +\left\vert f^{\prime }(b)\right\vert \right] .  \notag
\end{eqnarray}

\begin{proof}
Since $K$ is $\varphi -$convex with respect to $\varphi $, for every $t\in %
\left[ 0,1\right] $, we have $a+te^{i\theta }(b-a)\in K$. Integrating by
parts implies that%
\begin{eqnarray}
&&\int_{0}^{\frac{1}{2}}tf^{\prime }(a+te^{i\varphi }(b-a))dt+\int_{\frac{1}{%
2}}^{1}(t-1)f^{\prime }(a+te^{i\varphi }(b-a))dt  \notag \\
&&  \label{7} \\
&=&\left[ \frac{tf(a+te^{i\varphi }(b-a))}{e^{i\varphi }(b-a)}\right] _{0}^{%
\frac{1}{2}}+\left[ \frac{(t-1)f(a+te^{i\varphi }(b-a))}{e^{i\varphi }(b-a)}%
\right] _{\frac{1}{2}}^{1}  \notag \\
&&-\frac{1}{e^{i\varphi }(b-a)}\int_{0}^{1}f(a+te^{i\varphi }(b-a))dt  \notag
\\
&&  \notag \\
&=&\frac{1}{e^{i\varphi }(b-a)}f\left( \frac{2a+e^{i\varphi }(b-a)}{2}%
\right) -\frac{1}{e^{2i\varphi }(b-a)^{2}}\int_{a}^{a+e^{i\varphi
}(b-a)}f(x)dt.  \notag
\end{eqnarray}%
By $\varphi -$convexity of $\left\vert f^{\prime }\right\vert $, we have%
\begin{eqnarray*}
&&\left\vert \frac{1}{e^{i\varphi }(b-a)}\int_{a}^{a+e^{i\varphi
}(b-a)}f(x)dx-f\left( \frac{2a+e^{i\varphi }(b-a)}{2}\right) \right\vert  \\
&\leq &e^{i\varphi }(b-a)\left[ \int_{0}^{\frac{1}{2}}t\left\vert f^{\prime
}(a+te^{i\varphi }(b-a))\right\vert dt+\int_{\frac{1}{2}}^{1}(1-t)\left\vert
f^{\prime }(a+te^{i\varphi }(b-a))\right\vert dt\right]  \\
&\leq &e^{i\varphi }(b-a)\left[ \int_{0}^{\frac{1}{2}}t\left[
(1-t)\left\vert f^{\prime }(a)\right\vert +t\left\vert f^{\prime
}(b)\right\vert \right] dt+\int_{\frac{1}{2}}^{1}(1-t)\left[ (1-t)\left\vert
f^{\prime }(a)\right\vert +t\left\vert f^{\prime }(b)\right\vert \right] dt%
\right]  \\
&\leq &e^{i\varphi }(b-a)\left[ \frac{\left\vert f^{\prime }(a)\right\vert
+\left\vert f^{\prime }(b)\right\vert }{8}\right] .
\end{eqnarray*}%
The proof is completed.
\end{proof}
\end{theorem}

\begin{theorem}
\label{tt5} Under the assumptions of Theorem \ref{tt3}. Then the following
inequality holds:%
\begin{eqnarray}
&&\left\vert \frac{1}{e^{i\varphi }(b-a)}\int_{a}^{a+e^{i\varphi
}(b-a)}f(x)dx-f\left( \frac{2a+e^{i\varphi }(b-a)}{2}\right) \right\vert
\notag \\
&&  \label{8} \\
&\leq &\frac{e^{i\varphi }(b-a)}{16}\left( \frac{4}{p+1}\right) ^{\frac{1}{p}%
}\left[ \left( 3\left\vert f^{\prime }(a)\right\vert ^{\frac{p}{p-1}%
}+\left\vert f^{\prime }(b)\right\vert ^{\frac{p}{p-1}}\right) ^{\frac{p-1}{p%
}}+\left( \left\vert f^{\prime }(a)\right\vert ^{\frac{p}{p-1}}+3\left\vert
f^{\prime }(b)\right\vert ^{\frac{p}{p-1}}\right) ^{\frac{p-1}{p}}\right] .
\notag
\end{eqnarray}

\begin{proof}
From H\"{o}lder's inequality and by using (\ref{7}), we have%
\begin{eqnarray*}
&&\left\vert \frac{1}{e^{i\varphi }(b-a)}\int_{a}^{a+e^{i\varphi
}(b-a)}f(x)dx-f(\frac{2a+e^{i\varphi }(b-a)}{2})\right\vert  \\
&\leq &e^{i\varphi }(b-a)\left[ \int_{0}^{\frac{1}{2}}t\left\vert f^{\prime
}(a+te^{i\varphi }(b-a))\right\vert dt+\int_{\frac{1}{2}}^{1}(1-t)\left\vert
f^{\prime }(a+te^{i\varphi }(b-a))\right\vert dt\right]  \\
&\leq &e^{i\varphi }(b-a)\left( \int_{0}^{\frac{1}{2}}t^{p}dt\right) ^{\frac{%
1}{p}}\left( \int_{0}^{\frac{1}{2}}\left\vert f^{\prime }(a+te^{i\varphi
}(b-a))\right\vert ^{\frac{p}{p-1}}dt\right) ^{\frac{p-1}{p}} \\
&&+e^{i\varphi }(b-a)\left( \int_{\frac{1}{2}}^{1}\left( 1-t\right)
^{p}dt\right) ^{\frac{1}{p}}\left( \int_{\frac{1}{2}}^{1}\left\vert
f^{\prime }(a+te^{i\varphi }(b-a))\right\vert ^{\frac{p}{p-1}}dt\right) ^{%
\frac{p-1}{p}} \\
&\leq &\frac{e^{i\varphi }(b-a)}{2^{1+\frac{1}{p}}(p+1)^{\frac{1}{p}}}\left[
\int_{0}^{\frac{1}{2}}\left[ (1-t)\left\vert f^{\prime }(a)\right\vert ^{%
\frac{p}{p-1}}+t\left\vert f^{\prime }(b)\right\vert ^{\frac{p}{p-1}}\right]
dt\right] ^{\frac{p-1}{p}} \\
&&+\frac{e^{i\varphi }(b-a)}{2^{1+\frac{1}{p}}(p+1)^{\frac{1}{p}}}\left[
\int_{\frac{1}{2}}^{1}\left[ (1-t)\left\vert f^{\prime }(a)\right\vert ^{%
\frac{p}{p-1}}+t\left\vert f^{\prime }(b)\right\vert ^{\frac{p}{p-1}}\right]
dt\right] ^{\frac{p-1}{p}} \\
&=&\frac{e^{i\varphi }(b-a)}{16}\left( \frac{4}{p+1}\right) ^{\frac{1}{p}}%
\left[ \left( 3\left\vert f^{\prime }(a)\right\vert ^{\frac{p}{p-1}%
}+\left\vert f^{\prime }(b)\right\vert ^{\frac{p}{p-1}}\right) ^{\frac{p-1}{p%
}}+\left( \left\vert f^{\prime }(a)\right\vert ^{\frac{p}{p-1}}+3\left\vert
f^{\prime }(b)\right\vert ^{\frac{p}{p-1}}\right) ^{\frac{p-1}{p}}\right]
\end{eqnarray*}%
which completes the proof.
\end{proof}
\end{theorem}

\begin{theorem}
\label{tt6} Under the assumptions of Theorem \ref{tt2}. Then, the following
inequality holds:%
\begin{eqnarray}
&&\left\vert \frac{1}{e^{i\varphi }(b-a)}\int_{a}^{a+e^{i\varphi
}(b-a)}f(x)dx-f(\frac{2a+e^{i\varphi }(b-a)}{2})\right\vert  \notag \\
&&  \label{9} \\
&\leq &\frac{e^{i\varphi }(b-a)}{4}\left( \frac{4}{p+1}\right) ^{\frac{1}{p}}%
\left[ \left\vert f^{\prime }(a)\right\vert +\left\vert f^{\prime
}(b)\right\vert \right] .  \notag
\end{eqnarray}

\begin{proof}
We consider the inequality (\ref{8}) i.e%
\begin{eqnarray}
&&\left\vert \frac{1}{e^{i\varphi }(b-a)}\int_{a}^{a+e^{i\varphi
}(b-a)}f(x)dx-f(\frac{2a+e^{i\varphi }(b-a)}{2})\right\vert   \notag \\
&&  \label{10} \\
&\leq &\frac{e^{i\varphi }(b-a)}{16}\left( \frac{4}{p+1}\right) ^{\frac{1}{p}%
}\left[ \left( 3\left\vert f^{\prime }(a)\right\vert ^{\frac{p}{p-1}%
}+\left\vert f^{\prime }(b)\right\vert ^{\frac{p}{p-1}}\right) ^{\frac{p-1}{p%
}}+\left( \left\vert f^{\prime }(a)\right\vert ^{\frac{p}{p-1}}+3\left\vert
f^{\prime }(b)\right\vert ^{\frac{p}{p-1}}\right) ^{\frac{p-1}{p}}\right] .
\notag
\end{eqnarray}%
Let $a_{1}=3\left\vert f^{\prime }(a)\right\vert ^{\frac{p}{p-1}}$, $%
b_{1}=\left\vert f^{\prime }(b)\right\vert ^{\frac{p}{p-1}}$, $%
a_{2}=\left\vert f^{\prime }(a)\right\vert ^{\frac{p}{p-1}}$, $%
b_{2}=3\left\vert f^{\prime }(b)\right\vert ^{\frac{p}{p-1}}$. Here $%
0<\left( p-1\right) /p<1$, for $p>1$. Using the fact that,%
\begin{equation*}
\sum\limits_{k=1}^{n}\left( a_{k}+b_{k}\right) ^{s}\leq
\sum_{k=1}^{n}a_{k}^{s}+\sum_{k=1}^{n}b_{k}^{s},
\end{equation*}%
for $\left( 0\leq s<1\right) $, $a_{1},a_{2},...,a_{n}\geq 0$, $%
b_{1},b_{2},...,b_{n}\geq 0$, we obtain%
\begin{eqnarray*}
&&\frac{e^{i\varphi }(b-a)}{16}\left( \frac{4}{p+1}\right) ^{\frac{1}{p}}%
\left[ \left( 3\left\vert f^{\prime }(a)\right\vert ^{\frac{p}{p-1}%
}+\left\vert f^{\prime }(b)\right\vert ^{\frac{p}{p-1}}\right) ^{\frac{p-1}{p%
}}+\left( \left\vert f^{\prime }(a)\right\vert ^{\frac{p}{p-1}}+3\left\vert
f^{\prime }(b)\right\vert ^{\frac{p}{p-1}}\right) ^{\frac{p-1}{p}}\right]  \\
&\leq &\frac{e^{i\varphi }(b-a)}{16}\left( \frac{4}{p+1}\right) ^{\frac{1}{p}%
}(3^{\frac{p-1}{p}}+1)\left[ \left\vert f^{\prime }(a)\right\vert
+\left\vert f^{\prime }(b)\right\vert \right]  \\
&\leq &\frac{e^{i\varphi }(b-a)}{16}\left( \frac{4}{p+1}\right) ^{\frac{1}{p}%
}4\left[ \left\vert f^{\prime }(a)\right\vert +\left\vert f^{\prime
}(b)\right\vert \right]
\end{eqnarray*}%
which completed proof.
\end{proof}
\end{theorem}

\begin{theorem}
\label{z} Let $f:\rightarrow (0,\infty )$ be a differentiable mapping on $%
K^{0}$. Assume $q\in
\mathbb{R}
$ with $q\geq 1$. If $\left\vert f^{\prime }\right\vert ^{q}$ is $\varphi -$%
convex function on the interval of real numbers $K^{0}$ (the interior of $K$%
) and $a,b\in K$ with $a<a+e^{i\varphi }(b-a)$. Then the following
inequality holds:%
\begin{eqnarray*}
&&\left\vert \frac{1}{e^{i\varphi }(b-a)}\int_{a}^{a+e^{i\varphi
}(b-a)}f(x)dx-f\left( \frac{2a+e^{i\varphi }(b-a)}{2}\right) \right\vert \\
&& \\
&\leq &\frac{e^{i\varphi }(b-a)}{8}\left[ \left( \frac{2\left\vert f^{\prime
}(a)\right\vert ^{q}+\left\vert f^{\prime }(b)\right\vert ^{q}}{3}\right) ^{%
\frac{1}{q}}+\left( \frac{\left\vert f^{\prime }(a)\right\vert
^{q}+2\left\vert f^{\prime }(b)\right\vert ^{q}}{3}\right) ^{\frac{1}{q}}%
\right] .
\end{eqnarray*}

\begin{proof}
From H\"{o}lder's inequality and by using (\ref{7}), we have%
\begin{eqnarray*}
&&\left\vert \frac{1}{e^{i\varphi }(b-a)}\int_{a}^{a+e^{i\varphi
}(b-a)}f(x)dx-f(\frac{2a+e^{i\varphi }(b-a)}{2})\right\vert  \\
&\leq &e^{i\varphi }(b-a)\left[ \int_{0}^{\frac{1}{2}}t\left\vert f^{\prime
}(a+te^{i\varphi }(b-a))\right\vert dt+\int_{\frac{1}{2}}^{1}(1-t)\left\vert
f^{\prime }(a+te^{i\varphi }(b-a))\right\vert dt\right]  \\
&\leq &e^{i\varphi }(b-a)\left( \int_{0}^{\frac{1}{2}}tdt\right) ^{\frac{1}{p%
}}\left( \int_{0}^{\frac{1}{2}}t\left\vert f^{\prime }(a+te^{i\varphi
}(b-a))\right\vert ^{q}dt\right) ^{\frac{1}{q}} \\
&&+e^{i\varphi }(b-a)\left( \int_{\frac{1}{2}}^{1}\left( 1-t\right)
dt\right) ^{\frac{1}{p}}\left( \int_{\frac{1}{2}}^{1}\left( 1-t\right)
\left\vert f^{\prime }(a+te^{i\varphi }(b-a))\right\vert ^{q}dt\right) ^{%
\frac{1}{q}} \\
&\leq &\frac{e^{i\varphi }(b-a)}{8^{\frac{1}{p}}}\left[ \int_{0}^{\frac{1}{2}%
}t\left[ (1-t)\left\vert f^{\prime }(a)\right\vert ^{q}+t\left\vert
f^{\prime }(b)\right\vert ^{q}\right] dt\right] ^{\frac{1}{q}} \\
&&+\frac{e^{i\varphi }(b-a)}{8^{\frac{1}{p}}}\left[ \int_{\frac{1}{2}%
}^{1}\left( 1-t\right) \left[ (1-t)\left\vert f^{\prime }(a)\right\vert
^{q}+t\left\vert f^{\prime }(b)\right\vert ^{q}\right] dt\right] ^{\frac{1}{q%
}} \\
&=&\frac{e^{i\varphi }(b-a)}{8}\left[ \left( \frac{2\left\vert f^{\prime
}(a)\right\vert ^{q}+\left\vert f^{\prime }(b)\right\vert ^{q}}{3}\right) ^{%
\frac{1}{q}}+\left( \frac{\left\vert f^{\prime }(a)\right\vert
^{q}+2\left\vert f^{\prime }(b)\right\vert ^{q}}{3}\right) ^{\frac{1}{q}}%
\right] .
\end{eqnarray*}%
The proof is completed.
\end{proof}
\end{theorem}

\begin{theorem}
Under the assumptions of Theorem \ref{z}. Then the following inequality
holds:%
\begin{eqnarray*}
&&\left\vert \frac{1}{e^{i\varphi }(b-a)}\int_{a}^{a+e^{i\varphi
}(b-a)}f(x)dx-f\left( \frac{2a+e^{i\varphi }(b-a)}{2}\right) \right\vert  \\
&& \\
&\leq &\frac{e^{i\varphi }(b-a)}{8}(\frac{2^{\frac{1}{q}}+1}{3^{\frac{1}{q}}}%
)\left[ \left\vert f^{\prime }(a)\right\vert +\left\vert f^{\prime
}(b)\right\vert \right] .
\end{eqnarray*}

\begin{proof}
We consider the inequality (\ref{8}) i.e%
\begin{eqnarray*}
&&\left\vert \frac{1}{e^{i\varphi }(b-a)}\int_{a}^{a+e^{i\varphi
}(b-a)}f(x)dx-f(\frac{2a+e^{i\varphi }(b-a)}{2})\right\vert  \\
&& \\
&\leq &\frac{e^{i\varphi }(b-a)}{8}\left[ \left( \frac{2\left\vert f^{\prime
}(a)\right\vert ^{q}+\left\vert f^{\prime }(b)\right\vert ^{q}}{3}\right) ^{%
\frac{1}{q}}+\left( \frac{\left\vert f^{\prime }(a)\right\vert
^{q}+2\left\vert f^{\prime }(b)\right\vert ^{q}}{3}\right) ^{\frac{1}{q}}%
\right] .
\end{eqnarray*}%
Let $a_{1}=2\left\vert f^{\prime }(a)\right\vert ^{q}/3$, $b_{1}=\left\vert
f^{\prime }(b)\right\vert ^{q}/3$, $a_{2}=\left\vert f^{\prime
}(a)\right\vert ^{q}/3$, $b_{2}=2\left\vert f^{\prime }(b)\right\vert ^{q}/3$%
. Here $0<1/q<1$, for $q\geq 1$. Using the fact that%
\begin{equation*}
\sum\limits_{k=1}^{n}\left( a_{k}+b_{k}\right) ^{s}\leq
\sum_{k=1}^{n}a_{k}^{s}+\sum_{k=1}^{n}b_{k}^{s},
\end{equation*}%
for $\left( 0\leq s<1\right) $, $a_{1},a_{2},...,a_{n}\geq 0$, $%
b_{1},b_{2},...,b_{n}\geq 0$, we obtain%
\begin{eqnarray*}
&&\frac{e^{i\varphi }(b-a)}{8}\left[ \left( \frac{2\left\vert f^{\prime
}(a)\right\vert ^{q}+\left\vert f^{\prime }(b)\right\vert ^{q}}{3}\right) ^{%
\frac{1}{q}}+\left( \frac{\left\vert f^{\prime }(a)\right\vert
^{q}+2\left\vert f^{\prime }(b)\right\vert ^{q}}{3}\right) ^{\frac{1}{q}}%
\right]  \\
&\leq &\frac{e^{i\varphi }(b-a)}{8}(\frac{2^{\frac{1}{q}}+1}{3^{\frac{1}{q}}}%
)\left[ \left\vert f^{\prime }(a)\right\vert +\left\vert f^{\prime
}(b)\right\vert \right] .
\end{eqnarray*}%
This concludes the proof.
\end{proof}
\end{theorem}

\section{Hermite-Hadamard type inequalities for quasi $\protect\varphi -$%
convex functions}

In this section, we prove some new inequalities of Hermite-Hadamard for
quasi $\varphi -$convex function as follows:

\begin{theorem}
Let $f:K\rightarrow (0,\infty )$ be a differentiable mapping on $K^{0}$. If $%
\left\vert f^{\prime }\right\vert $ is quasi $\varphi -$convex function on
the interval of real numbers $K^{0}$ (the interior of $K$) and $a,b\in K$
with $a<$ $a+e^{i\varphi }(b-a)$. Then the following inequality holds:%
\begin{eqnarray*}
&&\left\vert \frac{1}{e^{i\varphi }(b-a)}\int_{a}^{a+e^{i\varphi
}(b-a)}f(x)dx-\frac{f(a)+f(a+te^{i\varphi }(b-a))}{2}\right\vert  \\
&& \\
&\leq &\frac{e^{i\varphi }(b-a)}{4}\max \{\left\vert f^{\prime
}(a)\right\vert ,\left\vert f^{\prime }(b)\right\vert \}.
\end{eqnarray*}

\begin{proof}
By quasi $\varphi -$convexity of $\left\vert f^{\prime }\right\vert $ and by
using (\ref{4}), we have%
\begin{eqnarray*}
&&\left\vert \frac{1}{e^{i\varphi }(b-a)}\int_{a}^{a+e^{i\varphi
}(b-a)}f(x)dx-\frac{f(a)+f(a+te^{i\varphi }(b-a))}{2}\right\vert  \\
&\leq &\frac{e^{i\varphi }(b-a)}{2}\int_{0}^{1}\left\vert (1-2t)\right\vert
\left\vert f^{\prime }(a+te^{i\varphi }(b-a))\right\vert dt \\
&\leq &\frac{e^{i\varphi }(b-a)}{2}\max \{\left\vert f^{\prime
}(a)\right\vert ,\left\vert f^{\prime }(b)\right\vert
\}\int_{0}^{1}\left\vert (1-2t)\right\vert dt \\
&\leq &\frac{e^{i\varphi }(b-a)}{4}\max \{\left\vert f^{\prime
}(a)\right\vert ,\left\vert f^{\prime }(b)\right\vert \}
\end{eqnarray*}%
which completes the proof.
\end{proof}
\end{theorem}

\begin{theorem}
Let $f:K\rightarrow (0,\infty )$ be a differentiable mapping on $K^{0}$%
.Assume $p\in
\mathbb{R}
$ with $p>1$. If $\left\vert f^{\prime }\right\vert ^{p/p-1}$ is quasi $%
\varphi -$convex function on the interval of real numbers $K^{0}$ (the
interior of $K$) and $a,b\in K$ with $a<e^{i\varphi }(b-a)$. Then the
following inequality holds:%
\begin{eqnarray*}
&&\left\vert \frac{1}{e^{i\varphi }(b-a)}\int_{a}^{a+e^{i\varphi
}(b-a)}f(x)dx-\frac{f(a)+f(a+te^{i\varphi }(b-a))}{2}\right\vert \\
&\leq &\frac{e^{i\varphi }(b-a)}{2(p+1)^{\frac{1}{p}}}\left[ \max
\{\left\vert f^{\prime }(a)\right\vert ^{\frac{p}{p-1}},\left\vert f^{\prime
}(b)\right\vert ^{\frac{p}{p-1}}\}\right] ^{\frac{p-1}{p}}.
\end{eqnarray*}
\end{theorem}

\begin{proof}
By quasi $\varphi -$convexity of $\left\vert f^{\prime }\right\vert ^{p/p-1}$
and by using (\ref{4}), we have%
\begin{eqnarray*}
&&\left\vert \frac{1}{e^{i\varphi }(b-a)}\int_{a}^{a+e^{i\varphi
}(b-a)}f(x)dx-\frac{f(a)+f(a+te^{i\varphi }(b-a))}{2}\right\vert  \\
&\leq &\frac{e^{i\varphi }(b-a)}{2}\int_{0}^{1}\left\vert (1-2t)\right\vert
\left\vert f^{\prime }(a+te^{i\varphi }(b-a))\right\vert dt \\
&\leq &\frac{e^{i\varphi }(b-a)}{2}\left( \int_{0}^{1}\left\vert
(1-2t)\right\vert ^{p}dt\right) ^{\frac{1}{p}}\left( \int_{0}^{1}\left\vert
f^{\prime }(a+te^{i\varphi }(b-a))\right\vert ^{\frac{p-1}{p}}dt\right) ^{%
\frac{p}{p-1}} \\
&\leq &\frac{e^{i\varphi }(b-a)}{2}\left( \int_{0}^{1}\left\vert
(1-2t)\right\vert ^{p}dt\right) ^{\frac{1}{p}}\left( \int_{0}^{1}\max
\{\left\vert f^{\prime }(a)\right\vert ^{\frac{p}{p-1}},\left\vert f^{\prime
}(b)\right\vert ^{\frac{p}{p-1}}\}dt\right) ^{\frac{p}{p-1}} \\
&\leq &\frac{e^{i\varphi }(b-a)}{2(p+1)^{\frac{1}{p}}}\left[ \max
\{\left\vert f^{\prime }(a)\right\vert ^{\frac{p}{p-1}},\left\vert f^{\prime
}(b)\right\vert ^{\frac{p}{p-1}}\}\right] ^{\frac{p-1}{p}}
\end{eqnarray*}%
which completes the proof.
\end{proof}

\begin{theorem}
Let $f:K\rightarrow (0,\infty )$ be a differentiable mapping on $K^{0}$ . If
$\left\vert f^{\prime }\right\vert $ is quasi $\varphi -$convex function on
the interval of real numbers $K^{0}$ (the interior of $K$) and $a,b\in K$
with $a<$ $a+e^{i\varphi }(b-a)$. Then the following inequality holds:%
\begin{eqnarray*}
&&\left\vert \frac{1}{e^{i\varphi }(b-a)}\int_{a}^{a+e^{i\varphi
}(b-a)}f(x)dx-\frac{f(a+te^{i\varphi }(b-a))}{2}\right\vert \\
&& \\
&\leq &\frac{e^{i\varphi }(b-a)}{4}\max \{\left\vert f^{\prime
}(a)\right\vert ,\left\vert f^{\prime }(b)\right\vert \}.
\end{eqnarray*}

\begin{proof}
By quasi $\varphi -$convexity of $\left\vert f^{\prime }\right\vert $ and by
using (\ref{7}), we have%
\begin{eqnarray*}
&&\left\vert \frac{1}{e^{i\varphi }(b-a)}\int_{a}^{a+e^{i\varphi
}(b-a)}f(x)dx-f\left( \frac{2a+e^{i\varphi }(b-a)}{2}\right) \right\vert  \\
&\leq &e^{i\varphi }(b-a)\left[ \int_{0}^{\frac{1}{2}}t\left\vert f^{\prime
}(a+te^{i\varphi }(b-a))\right\vert dt+\int_{\frac{1}{2}}^{1}(1-t)\left\vert
f^{\prime }(a+te^{i\varphi }(b-a))\right\vert dt\right]  \\
&\leq &e^{i\varphi }(b-a)\max \{\left\vert f^{\prime }(a)\right\vert
,\left\vert f^{\prime }(b)\right\vert \}\left[ \int_{0}^{\frac{1}{2}%
}tdt+\int_{\frac{1}{2}}^{1}(1-t)dt\right]  \\
&\leq &\frac{e^{i\varphi }(b-a)}{4}\max \{\left\vert f^{\prime
}(a)\right\vert ,\left\vert f^{\prime }(b)\right\vert \}.
\end{eqnarray*}%
This concludes the proof.
\end{proof}
\end{theorem}

\begin{theorem}
Let $f:K\rightarrow (0,\infty )$ be a differentiable mapping on $K^{0}$%
.Assume $p\in
\mathbb{R}
$ with $p>1$. If $\left\vert f^{\prime }\right\vert ^{p/p-1}$ is quasi $%
\varphi -$convex function on the interval of real numbers $K^{0}$ (the
interior of $K$) and $a,b\in K$ with $a<e^{i\varphi }(b-a)$. Then the
following inequality holds:%
\begin{eqnarray*}
&&\left\vert \frac{1}{e^{i\varphi }(b-a)}\int_{a}^{a+e^{i\varphi
}(b-a)}f(x)dx-\frac{f(a+te^{i\varphi }(b-a))}{2}\right\vert  \\
&\leq &\frac{e^{i\varphi }(b-a)}{2(p+1)^{\frac{1}{p}}}\left[ \max
\{\left\vert f^{\prime }(a)\right\vert ^{\frac{p}{p-1}},\left\vert f^{\prime
}(b)\right\vert ^{\frac{p}{p-1}}\}\right] ^{\frac{p-1}{p}}.
\end{eqnarray*}
\end{theorem}

\begin{proof}
By quasi $\varphi -$convexity of $\left\vert f^{\prime }\right\vert ^{p/p-1}$
and by using (\ref{4}), we have%
\begin{eqnarray*}
&&\left\vert \frac{1}{e^{i\varphi }(b-a)}\int_{a}^{a+e^{i\varphi
}(b-a)}f(x)dx-\frac{f(a+te^{i\varphi }(b-a))}{2}\right\vert  \\
&\leq &e^{i\varphi }(b-a)\left[ \int_{0}^{\frac{1}{2}}t\left\vert f^{\prime
}(a+te^{i\varphi }(b-a))\right\vert dt+\int_{\frac{1}{2}}^{1}(1-t)\left\vert
f^{\prime }(a+te^{i\varphi }(b-a))\right\vert dt\right]  \\
&\leq &e^{i\varphi }(b-a)\left( \int_{0}^{\frac{1}{2}}t^{p}dt\right) ^{\frac{%
1}{p}}\left( \int_{0}^{\frac{1}{2}}\left\vert f^{\prime }(a+te^{i\varphi
}(b-a))\right\vert ^{\frac{p}{p-1}}dt\right) ^{\frac{p-1}{p}} \\
&&+e^{i\varphi }(b-a)\left( \int_{\frac{1}{2}}^{1}\left( 1-t\right)
^{p}dt\right) ^{\frac{1}{p}}\left( \int_{\frac{1}{2}}^{1}\left\vert
f^{\prime }(a+te^{i\varphi }(b-a))\right\vert ^{\frac{p}{p-1}}dt\right) ^{%
\frac{p-1}{p}} \\
&\leq &\frac{e^{i\varphi }(b-a)}{2(p+1)^{\frac{1}{p}}}\left[ \max
\{\left\vert f^{\prime }(a)\right\vert ^{\frac{p}{p-1}},\left\vert f^{\prime
}(b)\right\vert ^{\frac{p}{p-1}}\}\right] ^{\frac{p-1}{p}}
\end{eqnarray*}%
which completes the proof.
\end{proof}

\end{document}